\documentclass[graybox]{svmult}

\usepackage{mathptmx}       
\usepackage{helvet}         
\usepackage{courier}        
\usepackage{type1cm}        
\usepackage{makeidx}         
\usepackage{graphicx}        
\usepackage{multicol}        
\usepackage[bottom]{footmisc}

\usepackage{amsmath, amssymb}
\usepackage{mathrsfs}

\makeindex             

\newcommand{\R}{\mathbb{R}}
\newcommand{\C}{\mathbb{C}}
\newcommand{\Z}{\mathbb{Z}}

\newcommand{\cF}{{\mathcal F}}
\newcommand{\cI}{{\mathcal I}}
\newcommand{\cJ}{{\mathcal J}}
\newcommand{\cN}{{\mathcal N}}
\newcommand{\cS}{{\mathcal S}}

\newcommand{\Ga}{\Gamma}

\newcommand{\loc}{\mathrm{loc}}
\newcommand{\per}{\mathrm{per}}

\newcommand{\weakto}{\rightharpoonup}

\DeclareMathOperator{\essinf}{\mathrm{ess}\,\mathrm{inf}}

\DeclareMathOperator{\supp}{\mathrm{supp}\,}

\begin{document}

\title*{Schr\"odinger-type equations with sign-changing nonlinearities: a survey}
\author{Bartosz Bieganowski}
\institute{Nicolaus Copernicus University, Faculty of Mathematics and Computer Science, ul. Chopina 12/18, 87-100 Toruń, Poland, \email{bartoszb@mat.umk.pl}
}
%
%
\maketitle

\abstract*{We are looking for solutions to nonlinear Schr\"odinger-type equations of the form
$$
(-\varDelta)^{\alpha / 2} u (x) + V(x) u(x) = h (x,u(x)), \quad x \in \mathbb{R}^N,
$$
where $V : \mathbb{R}^N \rightarrow \mathbb{R}$ is an external potential (bounded, unbounded or singular) and the nonlinearity $h : \mathbb{R}^N \times \mathbb{R} \rightarrow \mathbb{R}$ may be sign-changing. The aim of this paper is to give an introduction to the problem and the variational approach based on the Nehari manifold technique, and to survey recent results concerning the existence of ground and bound states.}

\abstract{We are looking for solutions to nonlinear Schr\"odinger-type equations of the form
$$
(-\varDelta)^{\alpha / 2} u (x) + V(x) u(x) = h (x,u(x)), \quad x \in \mathbb{R}^N,
$$
where $V : \mathbb{R}^N \rightarrow \mathbb{R}$ is an external potential (bounded, unbounded or singular) and the nonlinearity $h : \mathbb{R}^N \times \mathbb{R} \rightarrow \mathbb{R}$ may be sign-changing. The aim of this paper is to give an introduction to the problem and the variational approach based on the Nehari manifold technique, and to survey recent results concerning the existence of ground and bound states.}

\section{Introduction}
\label{sec:1}

Let $\alpha \in (0, 2]$ and $N > \alpha$. We consider the following (nonlocal) Schr\"odinger problem
\begin{equation}\label{eq:1.1}
(-\varDelta)^{\alpha / 2} u(x) + V(x) u(x) = f(x,u(x)) - K(x) |u(x)|^{p-2} u (x), \quad x \in \R^N,
\end{equation}
where $u \in H^{\alpha / 2} (\R^N)$ and $(-\varDelta)^{\alpha / 2}$ is the fractional Laplacian (for $\alpha = 2$ we take $-\varDelta$). We assume that $\alpha \in (0,2]$ and $N > \alpha$. The problem \eqref{eq:1.1} arises in many various branches of mathematical physics and nonlinear topics. The local case ($\alpha = 2$) of the equation \eqref{eq:1.1} finds applications in nonlinear optics, where the light propagation in photonic crystals is studied. The potential $V$ describes the nanostructure of the material and the nonlinearity is responsible for its polarization. In particular, when the nonlinearity is sign-changing, the material is a mixture of focusing and defocusing materials. The material may have a periodic structure (then $V$ is a $\mathbb{Z}^N$-periodic function) and may have a defect (then $V$ is a close-to-periodic potential). Singular and unbounded potentials are also intensively studied. The fractional case ($0 < \alpha < 2$) has been introduced to describe the propagation dynamics of wave pockets in the presence of harmonic potential and also for the free particle. Such an equations was also studied in the quantum scattering problem. For more details of the physical motivation we refer to e.g. \cite{Akozbek, NonKerrBook, Buryak, Doerfler, GoodmanWinsteinJNS2001, Kuchment, Malomed, Pankov, ReedSimon, NonlinearPhotonicCrystals, TerraciniVerzini, YuPhysRevA, Longhi, Robinett, ZhangLiu, ZhangZhong}. The equation \eqref{eq:1.1} has been very intesively studied, see e.g. \cite{AlamaLi, AmbrosettiCeramiRuiz, AmbrosettiColorado, AmbrosettiRabinowitz, BartschDingPeriodic, BartschDancerWang, BartschDing, Bartsch, BelmonteBetiaPelinovsky, BenciGrisantiMeicheletti, BenciRabinowitz, BuffoniJeanStuart, ChabrowskiSzulkin2002, ChenZouCalPDE2013, Cheng, CotiZelati, doO, GilbargTrudinger, IkomaTanaka, Jeanjean, KryszSzulkin, LiSzulkin, LiTang, LiWangZeng, Liu, MaiaJDE2006, MederskiTMNA2014, MederskiNLS2014, MontefuscoPellacciSquassinaJEMS2008, PankovDecay, PengChenTang, Rabinowitz:1992, Sirakov, WeiWethARMA2008, WillemZou, Zhang, dAveniaMederski, Davila2, Davila, Dipierro, Fall2, Frank2, GuoMederski}.

The equation \eqref{eq:1.1} describes the behaviour of the so-called standing wave solution $\Psi : \R^N \times \R \rightarrow \C$,
$$
\Psi(x,t) = \E^{-\I \omega t} u(x)
$$
of the time-dependent (fractional) Schr\"odinger equation
$$
\I \frac{\Psi(x,t)}{\partial t} = (-\varDelta)^{\alpha / 2} \Psi(x,t) + (V(x)+\omega)\Psi(x,t) - h(x,|\Psi(x,t)|) \Psi (x,t),
$$
where $\omega \in \R$. The fractional equation was introduced by Laskin by expanding the Feynman path integral from the Brownian-like to the L\'evy-like quantum mechanical paths (\cite{Laskin2000, Laskin2002}). 

The fractional Laplacian $(-\varDelta)^{\alpha/2}$ of a function $\psi : \R^N \rightarrow \R$ is given by the Fourier multiplier $|\xi|^\alpha$, i.e.
$$
\cF \left( (-\varDelta)^{\alpha / 2} \psi \right) (\xi) := |\xi|^\alpha \hat{\psi} (\xi),
$$
where
$$
\mathcal{F} \psi(\xi) := \hat{\psi} (\xi) := \int_{\R^N} e^{-\mathrm{i} \xi \cdot x} \psi(x) \, dx
$$
denotes the usual Fourier transform. When $\psi$ is a rapidly decaying smooth function, the fractional Laplacian can be defined by the principal value of the following singular integral
$$
(-\varDelta)^{\alpha / 2} \psi(x) = C_{N,\alpha} P.V. \int_{\R^N} \frac{\psi(x)-\psi(y)}{|x-y|^{N+\alpha}} \, dy,
$$
where $C_{N,\alpha} > 0$ is some normalization constant. We recall that
$$
P.V. \int_{\R^N} \frac{\psi(x)-\psi(y)}{|x-y|^{N+\alpha}} \, dy := \lim_{\varepsilon \to 0^+} \int_{|x-y| \geq \varepsilon} \frac{\psi(x)-\psi(y)}{|x-y|^{N+\alpha}} \, dy.
$$ 
These definitions are equivalent on $L^2 (\R^N)$, i.e. they give operators with common domain on $L^2 (\R^N)$ and they coincide on this domain (\cite{Kwasnicki}). 

Let us remind the definition of the fractional Sobolev space. For $0 < \alpha < 2$ we put
$$
H^{\alpha / 2} (\R^N) := \left\{ u \in L^2 (\R^N) \ : \ \int_{\R^N} |\xi|^\alpha |\hat{u}(\xi)|^2 \, d\xi + \int_{\R^N} |u(x)|^2 \, dx < \infty \right\}.
$$ 
It is a Hilbert space endowed with the scalar product
$$
H^{\alpha / 2} (\R^N) \times H^{\alpha / 2} (\R^N) \ni (u,v) \mapsto \int_{\R^N} |\xi|^\alpha \hat{u}(\xi) \overline{\hat{v}(\xi)} \, d\xi + \int_{\R^N} u(x)v(x) \, dx \in \R.
$$
The space $H^{\alpha / 2} (\R^N)$ is also denoted by $W^{\alpha / 2, 2} (\R^N)$. In the local case ($\alpha = 2$) we are working on the classical Sobolev space $H^1(\R^N) = W^{1,2} (\R^N)$ defined by
$$
H^1 (\R^N) := \left\{ u \in L^2 (\R^N) \ : \ \nabla u \in L^2 (\R^N) \right\}
$$
with the scalar product
$$
H^1 (\R^N) \times H^1(\R^N) \ni (u,v) \mapsto \int_{\R^N} \nabla u(x) \cdot \nabla v(x) \, dx + \int_{\R^N} u(x)v(x) \, dx \in \R.
$$
Recall also that there is a continuous embedding $H^{\alpha / 2} (\R^N) \subset L^t (\R^N)$ for all $t \in [2,2^*_\alpha]$, where $2^*_\alpha := \frac{2N}{N-\alpha}$. Moreover the embedding $H^{\alpha / 2} (\R^N) \subset L^t_{\loc} (\R^N)$ is compact for $t \in [2, 2^*_\alpha)$. We refer to \cite{Tartar} for more facts about Sobolev spaces. 

We will use the symbol $| \cdot |_k$ to denote the usual $L^k (\R^N)$-norm.

In what follows, the nonlinear term $f$ satisfies the following conditions.

\begin{enumerate}
\item[(F1)] $f : \R^N \times \R \rightarrow \R$ is measurable and $\mathbb{Z}^N$-periodic in $x \in \R^N$, continuous in $u \in \R$ for a.e. $x \in \R^N$ (i.e. $f$ is the Carath\'{e}odory function) and there are $c > 0$ and $2 < q < p < 2^*_\alpha$ such that
$$
|f(x,u)| \leq c (1+|u|^{p-1}) \mbox{ for all } u \in \R \mbox{ and a.e. } x \in \R^N.
$$
\item[(F2)] $f(x,u) = o(u)$ uniformy in $x$ as $|u|\to 0^+$.
\item[(F3)] $F(x,u)/|u|^q \to \infty$ uniformly in $x$ as $|u|\to\infty$, where $F(x,u) := \int_0^u f(x,s) \, ds$ is the primitive of $f$ with respect to $u$.
\item[(F4)] $\R \setminus \{0\} \ni u \mapsto f(x,u)/|u|^q \in \R$ is strictly increasing on $(-\infty,0)$ and on $(0,\infty)$.
\end{enumerate}

Observe that these assumptions imply that for every $\varepsilon > 0$ there is $C_\varepsilon > 0$ such that
\begin{equation}\label{eps}
|f(x,u)| \leq \varepsilon |u| + C_\varepsilon |u|^{p-1}
\end{equation}
for a.e. $x \in \R^N$ and all $u \in \R$.

On $K$ we impose the following condition:
\begin{enumerate}
\item[($K$)] $K \in L^\infty (\R^N)$ is $\mathbb{Z}^N$-periodic in $x \in \R^N$ and $K(x) \geq 0$ for a.e. $x \in \R^N$. 
\end{enumerate}

The nonlinearity $(x,u) \mapsto f(x,u)-K(x)|u|^{q-2}u$ does not satisfy the Ambrosetti-Rabinowitz-type condition. Moreover, it can be sign-changing, for example - consider $f(x,u) = |u|^{p-2}u$ and $K \equiv 1$, where $2 < q < p < 2^*_\alpha$.

The paper is organized as follows. The second chapter contains the general, variational theorem, which describes the geometry of the energy functional and allows us to find a bounded minimizing sequence. Sections \ref{sec:3}, \ref{sec:4}, \ref{sec:5} contain results concerning close-to-periodic, coercive and singular potentials, respectively. The sixth section presents two versions of the Palais-Smale sequences decomposition, which is needed to the analysis of the minimizing sequence when the potential is close-to-periodic or singular.

\section{Variational setting}
\label{sec:2}

Consider a (real) Hilbert space $(E, \langle \cdot, \cdot \rangle)$ with the norm $\| \cdot \|$ induced by the scalar product, i.e.
$$
\|u\|^2 := \langle u, u \rangle, \quad u \in E.
$$
Let $\cJ : E \rightarrow \R$ be a (nonlinear) functional of the general form
$$
\cJ (u) = \frac{1}{2} \|u\|^2 - \cI (u),
$$
where $\cI : E \rightarrow \R$ is of $C^1$-class and $\cI(0)=0$. The Nehari manifold is given by
$$
\cN = \{ u \in E \setminus \{0\} \ : \ \cJ'(u)(u) = 0 \}.
$$
Note that the condition $\cJ'(u)(u) = 0$ is equivalent to the following one
$$
\|u\|^2 = \cI'(u)(u).
$$
Moreover the set $\cN$ contains all nontrivial critical points of $\cJ$. The following general theorem is crucial to obtain existence results, in fact all assumptions in the following theorem describe the geometry of the functional $\cJ$ needed to obtain a bounded Palais-Smale sequence. 

\begin{theorem}[{\cite[Theorem 2.1]{BieganowskiMederski}}]\label{abstract}
Suppose that the following conditions hold:
\begin{enumerate}
\item[(J1)] there is a radius $r > 0$ such that $a := \inf_{\|u\|=r} \cJ(u) > 0 = \cJ(0)$;
\item[(J2)] there is $q \geq 2$ such that $\cI(t_n u_n) / t_n^q \to \infty$ for any $t_n \to \infty$ and $u_n \to u \neq 0$ as $n \to \infty$;
\item[(J3)] for all $t \in (0,\infty) \setminus \{1\}$ and $u \in \cN$ the following inequality holds true
$$
\frac{t^2-1}{2} \cI'(u)(u) - \cI(tu)+\cI(u) < 0;
$$
\item[(J4)] $\cJ$ is coercive on $\cN$, i.e. $\cJ(u_n) \to \infty$ if $\{ u_n \} \subset \cN$ is such that $\|u_n\| \to \infty$.
\end{enumerate}
Then $\inf_{\cN} \cJ > 0$ and there is a bounded Palais-Smale sequence for $\cJ$ on $\cN$, i.e. there is a sequence $\{ u_n \} \subset \cN$ such that
$$
\cJ(u_n) \to \inf_{\cN} \cJ \quad \mathrm{and} \quad \cJ'(u_n) \to 0.
$$
\end{theorem}

Note that, taking into account (J1) and (J2) we can easily check that $\cJ$ has the classical mountain pass geometry. Hence we are able to find a Palais-Smale sequence on the mountain pass level, however we do not know whether it is a bounded sequence and contained in $\cN$. In order to get the boundedness we assume in (J4) that $\cJ$ is coercive on $\cN$, which is, in applications, a weaker requirement than the classical Ambrosetti-Rabinowitz condition.

\begin{remark}\label{rem-abstr}
\begin{enumerate}
\item[(a)] In order to get (J3) it is sufficient to check
\begin{equation}
(1-t) (t \cI'(u)(u) - \cI'(tu)(u)) > 0
\end{equation}
for any $t \in (0,\infty) \setminus \{1\}$ and $u \in E$ such that $\cI'(u)(u) > 0$. 
\item[(b)] (J3) is equivalent to the condition: for $u \in \cN$, the point $t=1$ is the unique maximum of 
$$
(0,\infty) \ni t \mapsto \cJ(tu)  \in \R.
$$
\end{enumerate}
\end{remark}

\begin{proof}[Theorem \ref{abstract}]
For a given $u \neq 0$ we define a map
$$
\varphi : [0,\infty) \rightarrow \R
$$
by the formula $\varphi(t) = \cJ(tu)-\cJ(u)$. Obviously $\varphi$ is of $C^1$-class and $\varphi(1) = 0$. Moreover $\varphi$ is given by
$$
\varphi(t) = \frac{t^2-1}{2} \cI'(u)(u) - \cI(tu)+\cI(u)
$$
provided that $u \in \cN$. In view of (J1) and (J2) we easily obtain that
$$
\varphi(0) = - \cJ(u) < \varphi \left( \frac{r}{\|u\|} \right)
$$
and $\varphi(t) \to -\infty$ as $t \to \infty$. Therefore there is a maximum point $t(u) > 0$ of $\varphi$ which is a critical point of $\varphi$, i.e. $\cJ'(t(u)u)u = 0$ and $t(u)u \in \cN$. From Remark \ref{rem-abstr}(b) we get that this point is unique. Define a map $\hat{m} : E \setminus \{0\} \rightarrow \cN$ by
$$
\hat{m}(u) = t(u)u.
$$
Take $u_n \to u_0 \neq 0$ and denote $t_n = t(u_n)$ for $n \geq 0$, so that $\hat{m}(u_n) = t_n u_n$. Assume that $t_n \to \infty$. Then
$$
o(1) = \cJ(u_n) / t_n^q \leq \cJ(\hat{m}(u_n)) / t_n^q = \frac12 \|u_n\|^{2} t_n^{2-q} - \cI(t_n u_n) / t_n^q \to -\infty
$$
as $n \to \infty$, thus we get a contradiction. Hence $\{ t_n \} \subset \R$ is bounded and we may assume that $t_n \to t_0 \geq 0$. Then
$$
\cJ(t(u_0)u_0) \geq \cJ(t_0 u_0) = \lim_{n\to\infty} \cJ(t_n u_n) \geq \lim_{n\to\infty} \cJ(t(u_0) u_n) = \cJ(t(u_0) u_0),
$$
thus $t_0 = t(u_0)$ and we obtained that $\hat{m}(u_n) \to \hat{m}(u_0)$, and $\hat{m}$ is continuous. Then $m := \hat{m} \big|_\cS$ is an homeomorphism, where
$$
\cS = \{ u \in E \ : \ \|u\|=1 \}.
$$
Moreover the inverse is given by $m^{-1} (v) = v / \|v\|$. Therefore
$$
c = \inf_{u \in \cS} (\cJ \circ m)(u) = \inf_{u \in \cN} \cJ(u) \geq \inf_{u \in \cN} \cJ \left( \frac{r}{\|u\|} u \right) \geq a > 0.
$$
Observe that $\cJ \circ m$ is the restriction of $\cJ \circ \hat{m}$ to the sphere $\cS$. We will show that $\hat{\Phi} := \cJ \circ \hat{m}$ is of $C^1$-class. 
Take any $w, z \in E$ with $w \neq 0$. Let $u = \hat{m}(w)$. Then $t(u) = \frac{\|u\|}{\|w\|}$ and $u = \frac{\|u\|}{\|w\|} w$. Let $\delta > 0$ be a small number such that 
$$
w_t := w + tz \in E \setminus \{0\}
$$
for $|t| < \delta$. Then $u_t := \hat{m} (w_t) = s_t w_t$ for $t \in (-\delta, \delta)$, where $s_t > 0$ and $s_0 = \frac{\|u\|}{\|w\|}$. As we already showed, $\hat{m}$ is continuous and therefore
$$
(-\delta, \delta) \ni t \mapsto s_t \in (0, \infty)
$$
is also continuous. Moreover
$$
\cJ(s_0w) = \sup_{s > 0} \cJ (s w) \geq \cJ(s_t w).
$$
Hence, in view of the mean value theorem,
\begin{align*}
\hat{\Phi}(w_t) - \hat{\Phi}(w) &= \cJ(u_t) - \cJ(u) = \cJ(s_t w_t) - \cJ(s_0 w) \\ &\leq \cJ(s_t w_t) - \cJ(s_t w) = \cJ'(s_t [w + \tau_t (w_t-w)]) s_t (w_t-w),
\end{align*}
where $\tau_t \in (0,1)$. Hence 
$$
\hat{\Phi}(w_t) - \hat{\Phi}(w) \leq s_0 \cJ'(u) tz + o(t) \quad \mbox{as} \ t \to 0.
$$
Similarly
$$
\hat{\Phi}(w_t) - \hat{\Phi}(w) \geq \cJ(s_0 w_t) - \cJ(s_0 w) = s_0 \cJ'(u) tz + o(t)
$$
as $t \to 0$. Hence
$$
(\cJ \circ \hat{m})'(w)(z) = s_0 \cJ'(u)(z) = \frac{\|\hat{m}(u)\|}{\|w\|} \cJ'(\hat{m}(w))(z).
$$
In particular $\cJ \circ \hat{m}$ is of $C^1$-class. Hence $\cJ \circ m : \cS \rightarrow \R$ is also of $C^1$-class and taking $v \in \cS$ we obtain that
$$
(\cJ \circ m)'(v)(z) = \| m (v) \| \cJ' (m(v)) (z)
$$
for $z \in T_v \cS$, where $T_v \cS$ denotes the tangent space to $\cS$ at point $v \in \cS$. Hence, in view of the Ekeland variational principle (\cite[Theorem 8.5]{Willem}) there is a minimizing sequence $\{ v_n \} \subset \cS$ such that
$$
\cJ(m(v_n)) \to c, \quad (\cJ \circ m)'(v_n) \to 0.
$$
Take $u_n = m(v_n) \in \cN$. While $\cJ'(u_n)(v_n) = 0$ we get
$$
(\cJ \circ m)'(v_n)(z) = \|u_n\| \cJ'(u_n)(z) = \|u_n\| \cJ'(u_n)(z+tv_n)
$$
for any $z \in T_{v_n} \cS$ and $t \in \R$. Hence
$$
\| (\cJ \circ m)'(v_n) \| = \|u_n\| \| \cJ'(u_n) \|.
$$
Since $u_n \in \cN$, we have $\|u_n\| \geq \eta$ for some $\eta > 0$ and in view of the coercivity we have $\| u_n \| \leq M$ for some $M > 0$. Hence $\{ u_n \}$ is a bounded Palais-Smale sequence for $\cJ$ on $\cN$.
\qed
\end{proof}

\section{Close-to-periodic potentials}
\label{sec:3}

In this section we will provide existence, nonexistence and multiplicity results concerning problem \eqref{eq:1.1} in the presence of an external, close-to-periodic potential. More precisely, we assume that
\begin{enumerate}
\item[(V1)] $V = V_{\loc} + V_{\per}$, where $V_{\per} \in L^\infty (\R^N)$ is $\Z^N$-periodic and $V_{\loc} \in L^\infty (\R^N) \cap L^{s} (\R^N)$ for some $s \geq \frac{N}{\alpha}$;
\item[(V2)] there holds
$$
\left\{ \begin{array}{ll}
\inf \sigma (-\varDelta + V(x)) > 0 & \quad \mbox{for } \alpha = 2, \\
\essinf_{x \in \R^N} V(x) > 0 & \quad \mbox{for } 0 < \alpha < 2.
\end{array} \right.
$$
\end{enumerate}

Under (V1) and (V2) the formula
$$
\langle u, v \rangle := \int_{\R^N} |\xi|^\alpha \hat{u}(\xi) \overline{\hat{v}(\xi)} \, d\xi + \int_{\R^N} V(x) u(x)v(x) \, dx, \quad \alpha \in (0,2)
$$
or
$$
\langle u,v \rangle := \int_{\R^N} \nabla u \cdot \nabla v + V(x) u(x)v(x) \, dx, \quad \alpha = 2
$$
induces a scalar product on $H^{\alpha / 2} (\R^N)$, which is equivalent to the classic one. Then the energy functional associated with our problem on $H^{\alpha / 2} (\R^N)$ is given by
$$
\cJ(u) = \frac12 \|u\|^2 -\int_{\R^N} F(x,u) - \frac{1}{q} K(x) |u|^q \, dx.
$$
The Nehari manifold is given by
$$
\cN = \left\{ u \in H^{\alpha / 2} (\R^N) \setminus \{0\} \ : \ \|u\|^2 = \int_{\R^N} f(x,u)u \, dx - \int_{\R^N} K(x) |u|^q \, dx \right\}.
$$

The first result reads as follows.

\begin{theorem}[{$\alpha = 2$: \cite[Theorem 1.1]{BieganowskiMederski}, $0<\alpha<2$: \cite[Theorem 1.1]{Bieganowski}}]\label{existence-per}
Suppose that (V1), (V2), (K) and (F1)--(F4) hold, and $V_{\loc} \equiv 0$ or $V_{\loc}(x) < 0$ for a.e. $x \in \R^N$. Then \eqref{eq:1.1} has a ground state, i.e. there is a nontrivial critical point $u$ of $\cJ$ such that $\cJ(u) = \inf_{\cN} \cJ$. 
\end{theorem}

\begin{proof}[Theorem \ref{existence-per}]
From Theorem \ref{abstract} there exists a bounded sequence $\{ u_n \} \subset \cN$ such that
$$
\cJ(u_n) \to \inf_{\cN} \cJ =: c > 0, \quad \cJ'(u_n) \to 0.
$$
Define
$$
\cJ_{\per} (u) := \cJ(u) - \frac12 \int_{\R^N} V_{\loc}(x) u^2 \, dx
$$
and
$$
c_{\per} := \inf_{\cN_{\per}} \cJ_{\per},
$$
where
$$
\cN_{\per} = \left\{ u \in H^{\alpha / 2} (\R^N) \setminus \{0\} \ : \ \cJ_{\per}' (u)(u) = 0 \right\}.
$$ 
By Theorem \ref{ThDecomposition-per} we have that
$$
\cJ (u_n) \to \cJ(u_0) + \sum_{k=1}^\ell \cJ_{\per} (w^k),
$$
where $w^k$ are critical points of $\cJ_\per$. If $V_{\loc} \equiv 0$ we have $\cJ = \cJ_{\per}$. If $u_0 = 0$, we have
$$
c + o(1) = \cJ (u_n) \to \sum_{k=1}^\ell \cJ_{\per} (w^k) \geq \ell c
$$
and therefore $\ell = 1$ and $w^1 \neq 0$ is a ground state. If $u_0 \neq 0$ we have
$$
c + o(1) = \cJ (u_n) \to \cJ (u_0) + \sum_{k=1}^\ell \cJ_{\per} (w^k) \geq (\ell+1) c
$$
and therefore $\ell = 0$ and $\cJ (u_n) \to \cJ (u_0) = c$, so $u_0$ is a ground state.

Suppose that $V_{\loc} (x) < 0$ for a.e. $x \in \R^N$. Then $c_{\per} > c$. Suppose that $u_0 = 0$. Therefore
$$
c + o(1) = \cJ (u_n) \to \sum_{k=1}^\ell \cJ_{\per} (w^k) \geq \ell c_{\per} > \ell c
$$
and $\ell = 0$. Thus $\cJ (u_n) \to 0 = c$ -- a contradiction. Therefore $u_0 \neq 0$ and observe that
$$
c + o(1) = \cJ (u_n) \to \cJ(u_0) + \sum_{k=1}^\ell \cJ_{\per} (w^k) \geq c + \ell c_{\per},
$$
and $\ell = 0$. It means that $\cJ (u_n) \to \cJ (u_0) = c$ and $u_0$ is a ground state.
\qed
\end{proof}

The nonexistence result is the following. It is a new result in the fractional setting, i.e. for $0 < \alpha < 2$.

\begin{theorem}[{$\alpha = 2$: \cite[Theorem 1.2]{BieganowskiMederski}}]\label{nonexistence-per}
Suppose that (V1), (K) and (F1)--(F4) hold, and $V_{\loc}(x) > 0$ for a.e. $x \in \R^N$, while (V2) is replaced by
$$
\left\{ \begin{array}{ll}
\inf \sigma (-\varDelta + V_{\per}(x)) > 0 & \quad \mbox{for } \alpha = 2, \\
\essinf_{x \in \R^N} V_{\per}(x) > 0 & \quad \mbox{for } 0 < \alpha < 2.
\end{array} \right.
$$ 
Then \eqref{eq:1.1} has no ground state solutions.
\end{theorem}

\begin{proof}[Theorem \ref{nonexistence-per}]
Suppose, by contradiction, that there is a ground state $u_0 \in \cN$ of $\cJ$. Let $t_{\per} > 0$ be such that $t_{\per} u_0 \in \cN_{\per}$. Since $V_{\loc}(x) > 0$ for a.e. $x \in \R^N$, we have that
$$
\int_{\R^N} V_{\loc}(x) u_0^2 \, dx > 0
$$
and therefore
$$
c_{\per} := \inf_{\cN_{\per}} \cJ_{\per} \leq \cJ_{\per} (t_{\per} u_0) < \cJ(t_{\per} u_0) \leq \cJ(u_0) = c.
$$
Fix any $u \in \cN_{\per}$ and for $y \in \mathbb{Z}^N$ let us denote $\tau_y u := u(\cdot - y)$. For each $y \in \mathbb{Z}^N$ let $t_y$ be a number such that $t_y \tau_y u \in \cN$. Now observe that
\begin{align*}
\cJ_{\per} (u) = \cJ_{\per}(\tau_y u) \geq \cJ_{\per}(t_y \tau_y u) &= \cJ(t_y \tau_y u) - \int_{\R^N} V_{\loc}(x) (t_y \tau_y u)^2 \, dx \\ &\geq c - \int_{\R^N} V_{\loc}(x) (t_y \tau_y u)^2 \, dx.
\end{align*}
We are going to show that 
$$
\int_{\R^N} V_{\loc}(x) (t_y \tau_y u)^2 \, dx \to 0.
$$
Indeed, note that
$$
\int_{\R^N} V_{\loc}(x) (t_y \tau_y u)^2 \, dx = t_y^2 \int_{\R^N} V_{\loc}(x + y) u^2 \, dx.
$$
In view of (V1) we easily get
$$
\int_{\R^N} V_{\loc}(x + y) u^2 \, dx \to 0
$$
as $|y|\to\infty$. Since $\cJ_{\per}$ is coercive on $\mathcal{N}_{\per}$, then
$$
\cJ_{\per} (t_y \tau_y u) = \cJ_{\per} (t_y u) \leq c_{\per}
$$
implies that $\{ t_y \}$ is bounded. Therefore
$$
\cJ_{\per} (u) \geq c - \int_{\R^N} V_{\loc}(x) (t_y \tau_y u)^2 \, dx \to c.
$$
Taking infimum over all $u \in \cN_{\per}$ we have a contradiction $c_{\per} \geq c$.
\qed
\end{proof}

Now we will provide the multiplicity result in the case $V_{\loc} \equiv 0$. Suppose that $u$ is a solution of \eqref{eq:1.1} and $k \in \mathbb{Z}^N$, observe that $u(\cdot - k)$ is also a solution, provided that $V_{\loc} \equiv 0$. Therefore all elements of the orbit 
\begin{equation}\label{Eq:Orbit}
\mathcal{O}(u) := \left\{ u( \cdot - k) \ : \ k \in \mathbb{Z}^N \right\}
\end{equation}
of $u$ under the $\mathbb{Z}^N$-action are solutions. Thus, we define that $u_1$ and $u_2$ are geometrically distinct if their orbits satisfy $\mathcal{O}(u_1) \cap \mathcal{O}(u_2) = \emptyset$.

\begin{theorem}[{\cite[Theorem 1.2]{Bieganowski}}]\label{ThMultiplicity}
Suppose that (V1), (V2), (K), (F1)--(F4) are satisfied, $V_{\loc} \equiv 0$ and suppose that $f$ is odd in $u$. Then \eqref{eq:1.1} admits infinitely many pairs $\pm u$ of geometrically distinct solutions.
\end{theorem}

Put $c = \inf_{\cN} \cJ > 0$ and $\beta = \inf_{\cN} \|u\| > 0$. Theorem \ref{existence-per} provides that $c$ is attained at some function in $\cN$. By $\tau_k$ we denote the $\mathbb{Z}^N$-action on $H^{\alpha / 2} (\R^N)$, i.e.
$$
\tau_k u = u(\cdot - k), \quad k \in \mathbb{Z}^N.
$$
Obviously, $\tau_k \tau_{-k} u = \tau_{-k} \tau_{k} u = u$.

\begin{lemma}\label{Lem:ScalarProductInvariance}
There holds
$$
\langle \tau_k u, v \rangle = \langle u, \tau_{-k} v \rangle
$$
for every $u,v \in H^{\alpha / 2} (\R^N)$ and $k \in \mathbb{Z}^N$.
\end{lemma}

\begin{proof}
For $\alpha = 2$ the observation is trivial. Let $0 < \alpha < 2$. Then
$$
\mathcal{F}(\tau_k u) \cdot \overline{\mathcal{F}(v)} = \mathcal{F}(u) \cdot \overline{\mathcal{F}(\tau_{-k} v)}.
$$
Therefore
$$
\int_{\R^N} |\xi|^\alpha \widehat{\tau_k u} (\xi) \overline{\hat{v} (\xi)} \, d\xi = \int_{\R^N} |\xi|^\alpha \hat{u} (\xi) \overline{\widehat{\tau_{-k} v} (\xi)} \, d\xi.
$$
Obviously, using a change of variables $x \mapsto x + k$ and $\Z^N$-periodicity of $V$ we obtain
$$
\int_{\R^N} V(x) (\tau_k u) v \, dx = \int_{\R^N} V(x) u (\tau_k v) \, dx
$$ 
and we conclude.
\qed
\end{proof}

\begin{remark}
For given $k \in \mathbb{Z}^N$, let us consider $\tau_k$ as an operator 
$$
\tau_k : H^{\alpha / 2} (\R^N) \rightarrow H^{\alpha / 2} (\R^N).$$
Then obviously $\tau_k$ is linear. Moreover
$$
\|\tau_k u\| = \|u\|,
$$
thus $\tau_k$ is a bounded operator and $\|\tau_k\| = 1$. Thus we may consider an adjoint operator 
$$
\tau_k^* : H^{\alpha / 2} (\R^N) \rightarrow H^{\alpha / 2} (\R^N).
$$
Lemma \ref{Lem:ScalarProductInvariance} implies that
$$
\tau_k^* = \tau_{-k}.
$$
Moreover $\tau_k$ is an isomorphism and $\tau_k^{-1} = \tau_{-k} = \tau_k^*$. Thus $\tau_k$ is an orthogonal operator.
\end{remark}

\begin{lemma}\label{J-inv}
Let $\alpha \in (0,2]$. The functional $\cJ$ is $\mathbb{Z}^N$-invariant.
\end{lemma}

\begin{proof}
Let us start with the trivial observation that if $u \in H^{\alpha / 2} (\R^N)$, and $w \in \mathcal{O}(u)$, then by Lemma \ref{Lem:ScalarProductInvariance}
$$
\| u \| = \| w\|.
$$
Indeed, $w = \tau_k u$ for some $k \in \mathbb{Z}^N$. Then
$$
\|w\|^2 = \langle w, w \rangle = \langle \tau_k u, \tau_k u \rangle = \langle u, \tau_{-k} \tau_k u \rangle = \langle u, u \rangle = \|u\|^2.
$$
Then
$$
\cJ (w) = \frac{1}{2} \|w\|^2 - \int_{\R^N} F(x,w) \, dx + \frac{1}{q} \int_{\R^N} K(x) |w|^{q} \, dx.
$$
Changing variables in the integrals $x \mapsto x+k$ and the $\mathbb{Z}^N$-periodicity of $F$ and $K$ in $x$ gives
$$
\int_{\R^N} F(x,w) \, dx = \int_{\R^N} F(x,u) \, dx, \quad \frac{1}{q} \int_{\R^N} K(x) |w|^{q} \, dx = \frac{1}{q} \int_{\R^N} K(x) |u|^{q} \, dx.
$$
Therefore $\cJ (w) = \cJ (u)$. 
\qed
\end{proof}

\begin{lemma}\label{Lem:NehariInvariance}
$\cN$ is $\mathbb{Z}^N$-invariant.
\end{lemma}

\begin{proof}
Suppose that $u \in \cN$. Then
$$
\cJ'(\tau_k u) (\tau_k u) = \| \tau_k u \|^2 - \int_{\R^N} f(x,\tau_k u)\tau_k u \, dx + \int_{\R^N} K(x) |\tau_k u|^q \, dx.
$$
By the $\mathbb{Z}^N$-periodicity of $f$ and $K$ we have
$$
\int_{\R^N} f(x,\tau_k u)\tau_k u \, dx = \int_{\R^N} f(x, u) u \, dx, \quad \int_{\R^N} K(x) |\tau_k u|^q \, dx=\int_{\R^N} K(x) |u|^q \, dx.
$$
Moreover, by Lemma \ref{Lem:ScalarProductInvariance} $\| \tau_k u \| = \|u\|$ and finally $\cJ'(\tau_k u) (\tau_k u) = 0$, which shows that $\mathcal{O}(u) \subset \cN$.
\qed
\end{proof}

\begin{remark}
Lemma \ref{Lem:ScalarProductInvariance} implies that the unit sphere $\cS$ is $\mathbb{Z}^N$-invariant.
\end{remark}

Recall that for each $u \in H^{\alpha / 2} (\R^N)$ there is a unique number $t(u) > 0$ such that $t(u)u \in \cN$ and moreover the function $m : \cS \rightarrow \cN$ given by $m(u)=t(u)u$ is a homeomorphism (see Theorem \ref{abstract}). The inverse $m^{-1} : \cN \rightarrow \cS$ is given by $m^{-1} (u) = u/\|u\|$.

\begin{lemma}\label{LemEquivariance}
We have that the following functions:
\begin{itemize}
\item $m : \cS \rightarrow \cN$,
\item $m^{-1} : \cN \rightarrow \cS$,
\item $\nabla \cJ : H^{\alpha / 2} (\R^N) \rightarrow H^{\alpha / 2} (\R^N)$,
\item $ \nabla(\cJ \circ m) : \cS \rightarrow H^{\alpha / 2} (\R^N)$
\end{itemize}
are $\mathbb{Z}^N$-equivariant.
\end{lemma}

\begin{proof} $ $ 
\begin{itemize}
\item \textbf{Equivariance of} $m$. \\
Take $u \in \cS$, since $\|u\| = \|\tau_k u\|$, we have $\tau_k u \in \cS$. There is a unique number $t=t(u)>0$ such that $m(u) = t(u)u \in \cN$. We claim that $t(u) \tau_k u \in \cN$. Indeed
$$
t(u) \tau_k u = \tau_k \left( t(u)u \right) = \tau_k m(u) \in \cN,
$$
by Lemma \ref{Lem:NehariInvariance}. Thus 
$$
m(\tau_k u) = t(u) \tau_k u = \tau_k m(u).
$$
\item \textbf{Equivariance of} $m^{-1}$. \\
Let $u \in \cN$. By Lemma \ref{Lem:NehariInvariance} we have that $\tau_k u \in \cN$. Observe that
$$
m^{-1} (\tau_k u) = \frac{\tau_k u}{\| \tau_k u\|} = \frac{\tau_k u}{\|u\|} = \tau_k \left( \frac{u}{\|u\|} \right) = \tau_k m^{-1} (u).
$$
\item \textbf{Equivariance of} $\nabla \cJ$. \\ It follows directly from Lemma \ref{J-inv}.
\item \textbf{Equivariance of} $\nabla(\cJ \circ m)$. \\
For $u \in \cS$ and $z \in T_u \cS$, by the proof of Theorem \ref{abstract} we know that
$$
(\cJ \circ m)'(u)(z) = \| m(u) \| \cJ'(m(u))(z).
$$
So take $u \in \cS$ and $z \in T_{\tau_k u} \cS$. Then $\tau_{-k} z \in T_u \cS$. Therefore
\begin{eqnarray*}
& & \langle \nabla (\cJ \circ m) (\tau_k u), z \rangle = (\cJ \circ m)'(\tau_k u)(z) = \| m(\tau_k u) \| \cJ'(m(\tau_k u))(z) \\
&=& \|\tau_k m(u) \| \cJ' (\tau_k (m(u))(z) = \| m(u) \| \langle \nabla \cJ (\tau_k (m(u)), z \rangle \\
&=& \|m(u)\| \langle \tau_k \nabla \cJ (m(u)), z \rangle = \|m(u)\| \langle \nabla \cJ (m(u)), \tau_{-k} z \rangle \\
&=& (\cJ \circ m)'(u)(\tau_{-k} z) = \langle \nabla (\cJ \circ m) (u), \tau_{-k} z \rangle = \langle \tau_{k} \nabla (\cJ \circ m) (u),  z \rangle.
\end{eqnarray*}
Thus $\tau_{k} \nabla (\cJ \circ m) (u) = \nabla (\cJ \circ m) (\tau_k u)$ for every $u \in \cS$. \qed
\end{itemize}
\end{proof}

\begin{lemma}
The function $m^{-1} : \cN \to \cS$ is Lipschitz continuous.
\end{lemma}

\begin{proof}
Let $u,v \in \cN$. Observe that
\begin{align*}
\| m^{-1} (u) - m^{-1} (v) \| &= \left\| \frac{u-v}{\|u\|} + \frac{v \|v\|-v\|u\|}{\|u\| \cdot \|v\|} \right\| = \left\| \frac{u-v}{\|u\|} + \frac{v (\|v\|-\|u\|)}{\|u\| \cdot \|v\|} \right\| \leq \\
&\leq \frac{\|u-v\|}{\|u\|} + \frac{\left| \|v\| - \|u\| \right|}{\|u\|} \leq \frac{2\|u-v\|}{\|u\|} \leq \frac{2}{\beta} \|u-v\|,
\end{align*}
where $\beta = \inf_{u \in \cN} \|u\| > 0$. \qed
\end{proof}

To show Theorem \ref{ThMultiplicity} we employ the method introduced by A. Szulkin and T. Weth in \cite{SzulkinWeth}. Put $\mathscr{C} = \{ u \in \cS \ : \ (\mathcal{J} \circ m)'(u) = 0 \}$. Choose a set $\mathcal{F} \subset \mathscr{C}$ such that $\mathcal{F} = - \mathcal{F}$ and for each orbit $\mathcal{O}(w)$ there is a unique representative $v \in \mathcal{F}$. To show Theorem \ref{ThMultiplicity} we need to show that $\mathcal{F}$ is infinite. Suppose by contradiction that
$$
\mathcal{F} \ \mbox{is finite.}
$$
Put
$$
\kappa = \inf \{ \|v-w\| \ : \ v,w \in \mathscr{C}, v\neq w \}
$$
and note that $\kappa > 0$ (see \cite[Lemma 2.13]{SzulkinWeth}). Hence $\mathscr{C}$ is a discrete set. The following lemma has the crucial role in the proof of the multiplicity result and originally has been proven in \cite[Lemma 2.14]{SzulkinWeth}. Since the nonlinear term is sign-changing we need only a slight modification of the proof - we include all the details for the reader's convenience.

\begin{lemma}\label{LemDiscreteness}
Let $d \geq c$. If $\{ v_n^1 \}, \{ v_n^2 \} \subset \cS$ are two Palais-Smale sequences for $\mathcal{J} \circ m$ such that $\mathcal{J}(m(v_n^i)) \leq d$, $i=1,2$, then 
$$
\|v_n^1 - v_n^2 \| \to 0
$$
or
$$
\liminf_{n\to\infty} \|v_n^1 - v_n^2\| \geq \rho (d) > 0
$$
where $\rho(d)$ depends only on $d$, but not on the particular choise of sequences.
\end{lemma}

\begin{proof}
Arguing as in \cite{SzulkinWeth} we have that $u_n^i = m(v_n^i)$, $i=1,2$, are Palais-Smale sequences for $\cJ$. Moreover they are bounded in $H^{\alpha / 2}(\R^N)$, since $\cJ$ is coercive on $\cN$. In particular, $\{ u_n^1 \}$ and $\{ u_n^2 \}$ are bounded in $L^2 (\R^N)$, say $|u_n^1|_2 + |u_n^2|_2 \leq M$ for some $M>0$.

\begin{itemize}
\item \textbf{Case 1:} Assume that $|u_n^1 - u_n^2 |_p \to 0$. Fix $\varepsilon > 0$. Then, by (F1), (F2) we have
\begin{align*}
\| u_n^1 - u_n^2\|^2 &= \mathcal{J}'(u_n^1)(u_n^1-u_n^2) - \mathcal{J}'(u_n^2)(u_n^1-u_n^2)  \\ 
&\quad + \int_{\R^N} \left[ f(x,u_n^1) - f(x,u_n^2) \right] (u_n^1 - u_n^2) \, dx \\
&\quad - \int_{\R^N} K(x) \left[ |u_n^1|^{q-2} u_n^1 - |u_n^2|^{q-2} u_n^2 \right] (u_n^1-u_n^2) \, dx \\
&\leq \varepsilon \| u_n^1 - u_n^2\| \\
&\quad + \int_{\R^N} \left[ \varepsilon (|u_n^1| + |u_n^2|) + C_\varepsilon (|u_n^1|^{p-1} + |u_n^2|^{p-1}) \right] |u_n^1 - u_n^2| \, dx \\
&\quad - \int_{\R^N} K(x) \left[ |u_n^1|^{q-2} u_n^1 - |u_n^2|^{q-2} u_n^2 \right] (u_n^1-u_n^2) \, dx \\
&\leq (1+C_0) \varepsilon \|u_n^1 - u_n^2\| + D_\varepsilon |u_n^1 - u_n^2|_p + C_1 |K|_\infty |u_n^1-u_n^2|_q^q
\end{align*}
for each $n \geq n_\varepsilon$ and some constants $C_0, C_1, D_\varepsilon > 0$. From our assumption we have that 
$$
D_\varepsilon |u_n^1 - u_n^2|_p \to 0.
$$
Observe that
$$
C_1 |K|_\infty |u_n^1 - u_n^2|_q^q \leq C_1 |K|_\infty |u_n^1 - u_n^2|_2^{\theta q} |u_n^1-u_n^2|_p^{(1-\theta)q},
$$
where $\theta \in (0,1)$ is such that
$
\frac{1}{q} = \frac{\theta }{2} + \frac{1-\theta}{p}.
$
Thus 
$$
C_1 |K|_\infty |u_n^1 - u_n^2|_q^q \leq C_1 |K|_\infty M^{\theta q} |u_n^1-u_n^2|_p^{(1-\theta)q} \to 0.
$$
Finally
\begin{align*}
\limsup_{n\to\infty} \| u_n^1 - u_n^2\|^2 &\leq \limsup_{n\to\infty} (1+C_0) \varepsilon \|u_n^1 - u_n^2\| \\
&\quad + \limsup_{n\to\infty} D_\varepsilon |u_n^1 - u_n^2|_p \\
 &\quad + \limsup_{n\to\infty} C_1 |K|_\infty |u_n^1-u_n^2|_q^q \\
 &= (1+C_0) \varepsilon \limsup_{n\to\infty}  \|u_n^1 - u_n^2\|
\end{align*}
for every $\varepsilon > 0$. Therefore $\lim_{n\to\infty} \|u_n^1 - u_n^2\| = 0$. Finally
$$
\| v_n^1 - v_n^2 \| = \| m^{-1} (u_n^1) - m^{-1} (u_n^2) \| \leq L \| u_n^1 - u_n^2 \| \to 0,
$$
where $L > 0$ is a Lipschitz constant for $m^{-1}$.

\item \textbf{Case 2:} Assume that $|u_n^1 - u_n^2|_p \not\to 0$. By the Lions' lemma there are $y_n \in \R^N$ such that
$$
\int_{B(y_n, 1)} |u_n^1 - u_n^2|^2 \, dx = \max_{y \in \R^N} \int_{B(y, 1)} |u_n^1 - u_n^2|^2 \, dx \geq \varepsilon
$$
for some $\varepsilon > 0$. In view of Lemma \ref{LemEquivariance} we can assume that $\{ y_n \} \subset \R^N$ is bounded. Therefore, up to a subsequence we have
$$
u_n^1 \rightharpoonup u^1, \quad u_n^2 \rightharpoonup u^2
$$
where $u^1 \neq u^2$ and $\mathcal{J}'(u^1) = \mathcal{J}'(u^2) = 0$, and
$$
\| u_n^1 \| \to \alpha^1, \quad \| u_n^2 \| \to \alpha^2,
$$
where $\beta \leq \alpha^i \leq \nu (d) := \sup \{ \|u\| \ : \ u \in \cN, \ \mathcal{J}(u) \leq d \}$, $i=1,2$. Suppose that $u^1 \neq 0$ and $u^2 \neq 0$. Therefore $u^i \in \cN$ for $i=1,2$. Moreover
$$
v^i = m^{-1} (u^i) \in \cS, \quad i=1,2
$$
and $v^1 \neq v^2$. Then
$$
\liminf_{n\to\infty} \|v_n^1 - v_n^2\| = \liminf_{n\to\infty} \left\| \frac{u_n^1}{\| u_n^1\|} - \frac{u_n^2}{\| u_n^2\|} \right\| \geq \left\| \frac{u^1}{\alpha^1} - \frac{u^2}{\alpha^2} \right\| = \| \beta_1 v_1 - \beta_2 v_2 \|,
$$
where $\beta_i = \frac{\|u^i\|}{\alpha^i} \geq \frac{\beta}{\nu(d)}$, $i=1,2$. Obviously $\|v^1\|=\|v^2\|=1$, since $v^i \in \cS$ for $i=1,2$. Therefore
$$
\liminf_{n\to\infty} \|v_n^1 - v_n^2\| \geq \| \beta_1 v^1 - \beta_2 v^2\| \geq \min \{ \beta_1, \beta_2 \} \|v^1 - v^2\| \geq \frac{\beta \kappa}{\nu(d)}.
$$
If $u^2 = 0$, then $u^1 \neq u^2 = 0$. Therefore
$$
\liminf_{n\to\infty} \|v_n^1 - v_n^2\| = \liminf_{n\to\infty} \left\| \frac{u_n^1}{\| u_n^1\|} - \frac{u_n^2}{\| u_n^2\|} \right\| \geq \left\| \frac{u^1}{ \alpha^1} - \frac{u^2}{\alpha^2} \right\| =  \left\| \frac{u^1}{ \alpha^1} \right\| \geq \frac{\beta}{\nu(d)}.
$$
The case $u^1=0$ is similar, the proof is completed.\qed
\end{itemize}
\end{proof}

\begin{proof}{Theorem \ref{ThMultiplicity}}
The unit sphere $\cS \subset H^{\alpha / 2} (\R^N)$ is a Finsler $C^{1,1}$-manifold and by \cite[Lemma II.3.9]{Struwe}, $\cJ \circ m : \cS \rightarrow \R$ admits a pseudo-gradient vector field, i.e. there is a Lipschitz continuous map $\mathcal{H} : \cS \setminus \mathscr{C} \rightarrow T \cS$ satisfying
\begin{align*}
\mathcal{H} (w) &\in T_w \cS, \\
\| \mathcal{H} (w) \| &< 2 \| \nabla (\cJ \circ m) (w) \|, \\
\langle \mathcal{H}(w), \nabla (\cJ \circ m)(w) \rangle &> \frac{1}{2} \| \nabla (\cJ \circ m)(w)\|^2
\end{align*}
for $w \in S \setminus \mathscr{C}$. The obtained discreteness of Palais-Smale sequences (Lemma \ref{LemDiscreteness}) allows us to repeat the proof of Lemma 2.15, Lemma 2.16 and Theorem 1.2 from \cite{SzulkinWeth} in our case. We show that for every $n \in \mathbb{N}$ there is $v_n \in \cS$ such that 
$$(\cJ \circ m)'(v_n) = 0 \mbox{ and } \mathcal{J} (m(v_n)) = c_n,$$ where 
$$
c_n = \inf \{ d \in \mathbb{R} \ : \ \gamma \left( \{ v \in \cS \ : \ \mathcal{J}(m(v)) \leq d \} \right) \geq n \}
$$
is the Lusternik-Schnirelmann value and $\gamma$ denotes the Krasnoselskii genus of closed and symmetric subsets. We refer the reader to \cite{Struwe} for more informations about the Krasnoselskii genus and Lusternik-Schnirelmann values. Note that $\cJ'(m(v_n)) = 0$, hence $m(v_n) \in \cN$ is a nontrivial critical point of $\cJ$. Moreover $c_n < c_{n+1}$, thus we get a contradiction with the finiteness of $\cF$.
\qed
\end{proof}

\section{Coercive potentials}
\label{sec:4}

This section is devoted to the existence result for \eqref{eq:1.1} with coercive potentials, i.e. potentials satisfying the following assumption
\begin{enumerate}
\item[(V3)] $V \in C (\R^N)$ is such that
$$
\lim_{|x|\to\infty} V(x) = + \infty
$$
and $V_0 := \inf_{x \in \R^N} V(x) > 0$.
\end{enumerate}

Under (V3) we define the subspace of $H^{\alpha / 2} (\R^N)$ by
$$
E^{\alpha / 2} = \left\{ u \in H^{\alpha / 2} (\R^N) \ : \ \int_{\R^N} V(x) u^2 \, dx < \infty \right\} \subset H^{\alpha / 2} (\R^N).
$$
On $E^{\alpha / 2}$ the formula
$$
\langle u, v \rangle := \int_{\R^N} |\xi|^\alpha \hat{u}(\xi) \overline{\hat{v}(\xi)} \, d\xi + \int_{\R^N} V(x) uv \, dx, \quad \alpha \in (0,2)
$$
or
$$
\langle u,v \rangle := \int_{\R^N} \nabla u \cdot \nabla v + V(x) uv \, dx, \quad \alpha = 2
$$
induces a scalar product. Then the energy functional associated with our problem on $E^{\alpha / 2}$ is given by
$$
\cJ(u) = \frac12 \|u\|^2 -\int_{\R^N} F(x,u) - \frac{1}{q} K(x) |u|^q \, dx, \quad u \in E^{\alpha / 2}.
$$
The Nehari manifold is given by
$$
\cN = \left\{ u \in E^{\alpha / 2} \setminus \{0\} \ : \ \|u\|^2 = \int_{\R^N} f(x,u)u \, dx - \int_{\R^N} K(x) |u|^q \, dx \right\}.
$$

We will use the following variant of the Sobolev-Gagliardo-Nirenberg inequality.

\begin{lemma}[{\cite[Proposition II.3]{secchi}}]\label{SGN}
Let $r > 1$. Then there is a positive constant $C > 0$, such that for every function $u \in H^{\alpha / 2} (\R^N)$ there holds
$$
|u|_{r+1}^{r+1} \leq C \|u\|_{H^{\alpha/2} (\R^N)}^{\frac{(r-1)N}{\alpha}} |u|_2^{r+1-\frac{(r-1)N}{\alpha}},
$$
where $\| \cdot \|_{H^{\alpha/2} (\R^N)}$ denotes the usual $H^{\alpha/2}(\R^N)$-norm.
\end{lemma}

Our result reads as follows.

\begin{theorem}[{\cite[Theorem 1.3]{Bieganowski}}]\label{coercive}
Suppose that (V3), (K), (F1)--(F4) are satisfied. Then \eqref{eq:1.1} has a ground state, i.e. there is a nontrivial critical point $u$ of $\cJ$ such that $\cJ(u) = \inf_{\cN} \cJ$. 
\end{theorem}

\begin{proof}[Theorem \ref{coercive}]
From Theorem \ref{abstract} there exists a bounded sequence $ \{ u_n \} \subset \cN$ such that
$$
\cJ(u_n) \to \inf_{\cN} \cJ =: c > 0, \quad \cJ'(u_n) \to 0.
$$
Then we may assume that $u_n \weakto u_0$ in $E^{\alpha / 2}$ and $u_n \to u_0$ in $L^r_{\loc} (\R^N)$ for $2 \leq r <2^*_\alpha$. We can easily check that $\cJ ' (u_0) = 0$ (see e.g. the proof of \cite[Theorem 4.1(a)]{BieganowskiMederski}). We only need to check whether $u_0 \neq 0$. Observe that for $n \geq n_0$, using \eqref{eps}, we have
\begin{align*}
\frac{c}{2} &\leq \cJ(u_n) = \cJ(u_n) - \frac{1}{2} \cJ'(u_n)(u_n) \\
&= \frac{1}{2} \int_{\R^N} (f(x,u_n)u_n - 2 F(x,u_n)) \, dx - \left( \frac12 -\frac{1}{q} \right) \int_{\R^N} K(x) |u|^q \, dx \\
&\leq \frac{1}{2} \int_{\R^N} (f(x,u_n)u_n - 2 F(x,u_n)) \, dx \leq \frac12 \int_{\R^N} f(x,u_n)u_n \, dx \\
&\leq \frac{1}{2} \int_{\R^N} \varepsilon |u_n|^2 + C_\varepsilon |u_n|^p \, dx,
\end{align*}
where $n_0 \geq 1$ is large enough. In view of the boundedness of $\{ u_n \}$ and taking Lemma \ref{SGN} into account we get
$$
\frac{c}{2} \leq \frac{\varepsilon}{2} |u_n|_2^2 + D_\varepsilon |u_n|_2^{p-\frac{(p-2)N}{\alpha}}
$$
for some $D_\varepsilon > 0$. Take $\varepsilon \leq \frac{c}{2 (\sup \|u_n\|)^2}$. Then
$$
\frac{c}{2} \leq \frac{c |u_n|_2^2}{4 (\sup \|u_n\|)^2} + D |u_n|_2^{p-\frac{(p-2)N}{\alpha}}
$$
for some $D > 0$. While $\frac{|u_n|_2^2}{(\sup \|u_n\|)^2} \leq 1$ we get
$$
\frac{c}{2} \leq \frac{c}{4} + D |u_n|_2^{p-\frac{(p-2)N}{\alpha}}.
$$
Hence we obtain that
$$
|u_n|_2^2 \geq \left( \frac{1}{C_1} \exp \left( \frac{\alpha}{\alpha p -(p-2)N} \ln \frac{c}{4} \right) \right)^2 =: \tilde{c} > 0,
$$
where $C_1 > 0$. Take any radius $R > 0$ and write
$$
|u_n|_2^2 = \int_{|x| < R} |u_n|^2 \, dx + \int_{|x| \geq R} |u_n|^2 \, dx.
$$
Assume by contradiction that $u_0 = 0$ and, in particular, $u_n \to 0$ in $L^2_{\loc} (\R^N)$. Then for every $R > 0$ there is $n_0$ such that for all $n \geq n_0$ there holds
$$
\int_{|x|< R} |u_n|^2 \, dx \leq \frac{\tilde{c}}{2}
$$
and therefore
$$
\int_{|x| \geq R} |u_n|^2 \, dx \geq \frac{\tilde{c}}{2}.
$$
On the other hand
\begin{align*}
\frac{\tilde{c}}{2} &\leq \int_{|x| \geq R} |u_n|^2 \, dx = \int_{|x| \geq R} \frac{V(x)|u_n|^2}{V(x)} \, dx \leq \frac{1}{\inf_{|x| \geq R} V(x)} \int_{|x| \geq R} V(x)|u_n|^2 \, dx \\
&\leq \frac{\|u_n\|^2}{\inf_{|x| \geq R} V(x)} \leq \frac{\sup \|u_n\|^2}{\inf_{|x| \geq R} V(x)}.
\end{align*}
Taking $R > 0$ large enough we obtain a contradiction, since $V(x) \to \infty$ as $|x| \to \infty$. Hence $u_0 \neq 0$ and the proof is completed.
\qed
\end{proof}

\section{Singular potentials}\label{sec:5}

In this section we will provide existence and nonexistence results for the equation with a singular potential and we focus only on the fractional case, i.e. $0<\alpha<2$. We assume that the potential is of the form
$$
V(x) - \frac{\mu}{|x|^\alpha},
$$
where $V$ satisfies (V1) and (V2), and $\mu \in \R$. In what follows we will use the following quadratic form definition of the fractional Laplacian (\cite{Kwasnicki}), i.e.
$$
\langle (-\varDelta)^{\alpha / 2} u , \varphi \rangle = \frac{2^\alpha \Gamma \left( \frac{N+\alpha}{2} \right)}{2 \pi^{N/2} \left| \Gamma \left( - \frac{\alpha}{2} \right) \right|} \iint_{\R^N \times \R^N} \frac{(u(y)-u(x))(\varphi(y)-\varphi(x))}{|x-y|^{N+\alpha}} \, dx \, dy
$$
on $H^{\alpha / 2} (\R^N)$. In what follows we will denote
$$
c_{N,\alpha} := \frac{2^\alpha \Gamma \left( \frac{N+\alpha}{2} \right)}{2 \pi^{N/2} \left| \Gamma \left( - \alpha/2 \right) \right|} .
$$
On $H^{\alpha / 2} (\R^N)$ we define a norm
\begin{align*}
\|u\|^2 :&= \langle (-\varDelta)^{\alpha / 2} u , u \rangle + \int_{\R^N} V(x) u^2 \, dx \\ &= c_{N,\alpha} \iint_{\R^N \times \R^N} \frac{|u(x)-u(y)|^2}{|x-y|^{N+\alpha}} \, dx \, dy + \int_{\R^N} V(x) u^2 \, dx.
\end{align*}
Then the energy functional is given by
$$
\cJ(u) = \frac12 \|u\|^2 - \frac{\mu}{2} \int_{\R^N} \frac{u^2}{|x|^\alpha} \, dx - \int_{\R^N} F(x,u) - \frac{1}{q} K(x) |u|^q \, dx.
$$
The Nehari manifold in this setting is given in the usual way
$$
\cN = \{ u \in H^{\alpha / 2} (\R^N) \setminus \{0\} \ : \ \cJ'(u)(u) = 0 \}.
$$
Let us recall the fractional Hardy inequality, which is the main tool and allows us to deal with Hardy-type potentials.

\begin{lemma}[{\cite{FrankSeiringer}[Theorem 1.1]}] \label{Lem:HardyIneq}
There is $H_{N,\alpha} > 0$ such that for every $u \in H^{\alpha / 2} (\R^N)$ and $N > \alpha$ there holds
$$
\iint_{\R^N \times \R^N} \frac{|u(x)-u(y)|^2}{|x-y|^{N+\alpha}} \, dx \, dy \geq H_{N,\alpha} \int_{\R^N} \frac{|u(x)|^2}{|x|^{\alpha}} \, dx.
$$
\end{lemma}

Moreover, in view of \cite{FrankSeiringer}, the sharp constant $H_{N,\alpha}$ can be estabilished and it is equal
$$
H_{N,\alpha} = 2 \pi^{N/2} \frac{\Ga \left( \frac{N+\alpha}{4} \right)^2  | \Ga ( - \alpha / 2)|}{\Ga \left( \frac{N-\alpha}{4} \right)^2 \Ga \left( \frac{N+\alpha}{2}  \right)}.
$$
Define 
$$
\mu^* := H_{N,\alpha} c_{N,\alpha} = 2^{\alpha} \left( \frac{\Ga \left( \frac{N+\alpha}{4} \right) }{\Ga \left( \frac{N-\alpha}{4} \right) } \right)^2.
$$
Note that in the local case ($\alpha = 2$) we obtain
$$
\mu^* = 4 \left( \frac{\Gamma \left(\frac{N}{4} - \frac{1}{2}+1\right)}{\Gamma \left(\frac{N}{4} - \frac{1}{2}\right)} \right)^2 = 4 \left( \frac{N}{4} - \frac{1}{2} \right)^2 = \frac{(N-2)^2}{4}.
$$
The same constant has been obtained in \cite{GuoMederski} and it is the sharp constant in the Hardy inequality in $H^1 (\R^N)$. 

Now we can state the main result of this section.

\begin{theorem}[{\cite[Theorem 1.1]{Bieganowski-hardy}}]\label{existence-hardy}
Assume that (V1), (V2), (K), (F1)--(F4) are satisfied, $V_\loc \equiv 0$ or $V_\loc(x) < 0$ for a.e. $x \in \R^N$, and $0 \leq \mu < \mu^*$. Then \eqref{eq:1.1} has a ground state, i.e. there is a nontrivial critical point $u$ of $\cJ$ such that $\cJ(u) = \inf_{\cN} \cJ$.
\end{theorem}

The following fact is very useful to deal with the Hardy-type term and plays a very important role in the proof of the decomposition of minimizing sequences (see Section \ref{sec:6}: Theorem \ref{ThDecomposition}).

\begin{lemma}\label{hardyLemma}
If $|x_n|\to\infty$, then for any function $u \in H^{\alpha / 2}(\R^N)$,
$$
\int_{\R^N} \frac{1}{|x|^\alpha} |u(\cdot - x_n)|^2 \, dx \to 0.
$$
\end{lemma}

\begin{proof}
Let $\varphi_m \in C_0^\infty (\R^N)$ and $\varphi_m \to u$ in $H^{\alpha / 2} (\R^N)$ as $m \to \infty$. Take $R_m > 0$ large enough that
$$
\supp \varphi_m \subset B(0, R_m).
$$
Obviously, for any $m$ there is $n(m)$ such that $|x_{n(m)}|-R_m \geq m$ and $n(m)$ is increasing. We get
\begin{align*}
\int_{\R^N} \frac{1}{|x|^\alpha} | \varphi_m (\cdot - x_n)|^2 \, dx &= \int_{\R^N} \frac{1}{|x+x_n|^\alpha} |\varphi_m|^2 \, dx = \int_{B(0,R_m)} \frac{1}{|x+x_n|^\alpha} |\varphi_m|^2 \, dx \\
&\leq \frac{1}{(|x_n|-R_m)^\alpha} \int_{B(0,R_m)} |\varphi_m|^2 \, dx \leq \frac{1}{m^\alpha} \int_{\R^N} |\varphi_m|^2 \, dx \to 0
\end{align*}
We have
\begin{align*}
&\quad \int_{\R^N} \frac{1}{|x|^\alpha} | u(\cdot - x_n) |^2 \, dx \\
&\leq \int_{\R^N} \frac{1}{|x|^\alpha} | u(\cdot - x_n) - \varphi_m (\cdot - x_n) |^2 \, dx + \int_{\R^N} \frac{1}{|x|^\alpha} | \varphi_m (\cdot - x_n)|^2 \, dx \\ &= \int_{\R^N} \frac{1}{|x|^\alpha} | u(\cdot - x_n) - \varphi_m (\cdot - x_n) |^2 \, dx + o(1).
\end{align*}
In view of the fractional Hardy inequality we obtain
\begin{align*}
&\quad \int_{\R^N} \frac{| u(\cdot - x_n) - \varphi_m (\cdot - x_n)|^2}{|x|^\alpha}  \, dx \\
&\leq \frac{1}{H_{N,\alpha}} \iint_{\R^N \times \R^N} \frac{|u(x)- \varphi_m (x)- u(y) + \varphi_m(y)|^2}{|x-y|^{N+\alpha}} \, dx \, dy \to 0,
\end{align*}
since $\varphi_m \to u$ in $H^{\alpha / 2} (\R^N)$.
\qed
\end{proof}

\begin{lemma}\label{norm-eqv}
Let $0 \leq \mu < \mu^*$. There exists $0 < D_{N,\alpha,\mu} \leq 1$ such that for any function $u \in H^{\alpha / 2} (\R^N)$ 
$$
D_{N,\alpha,\mu} \|u\|^2 \leq \|u\|^2 - \mu \int_{\R^N} \frac{u^2}{|x|^\alpha} \, dx \leq \|u\|^2.
$$
\end{lemma}

\begin{proof}
We have
$$
\| u \|^2 = c_{N,\alpha} \iint_{\R^N \times \R^N} \frac{|u(x)-u(y)|^2}{|x-y|^{N+\alpha}} \, dx \, dy + \int_{\R^N} V(x) u^2 \, dx.
$$
Put
$$
\| u \|^2_\mu := \| u \|^2 - \mu \int_{\R^N} \frac{u^2}{|x|^\alpha} \, dx.
$$
In view of Lemma \ref{Lem:HardyIneq}
$$
c_{N,\alpha} \iint_{\R^N \times \R^N} \frac{|u(x)-u(y)|^2}{|x-y|^{N+\alpha}} \, dx \, dy \geq c_{N,\alpha} H_{N,\alpha} \int_{\R^N} \frac{|u(x)|^2}{|x|^\alpha} \, dx.
$$
 Then $\mu < \mu^*$ means that
$$
1 \geq 1 - \frac{\mu}{c_{N,\alpha} H_{N,\alpha}} > 0.
$$
Then
\begin{align} \label{Eq:2.1}
& c_{N,\alpha} \iint_{\R^N \times \R^N} \frac{|u(x)-u(y)|^2}{|x-y|^{N+\alpha}} \, dx \, dy - \mu \int_{\R^N} \frac{|u(x)|^2}{|x|^\alpha} \, dx \\ \geq& \left(1 - \frac{ \mu}{c_{N,\alpha} H_{N,\alpha} } \right) c_{N,\alpha} \iint_{\R^N \times \R^N}  \frac{|u(x)-u(y)|^2}{|x-y|^{N+\alpha}} \, dx \, dy \geq 0.\nonumber
\end{align}
On the other hand
\begin{equation} \label{Eq:2.2}
\| u\|_\mu^2 \geq \left( 1 - \frac{ \mu}{c_{N,\alpha} H_{N,\alpha} } \right) c_{N,\alpha} \iint_{\R^N \times \R^N}  \frac{|u(x)-u(y)|^2}{|x-y|^{N+\alpha}} \, dx \, dy.
\end{equation}
From (\ref{Eq:2.1}) we have
\begin{equation} \label{Eq:2.3}
\| u \|_\mu^2 \geq \int_{\R^N} V(x) u^2 \, dx.
\end{equation}
Combining (\ref{Eq:2.2}) and (\ref{Eq:2.3}) we obtain
\begin{align*}
\| u \|_\mu^2 &= \frac12 \|u\|_\mu^2 + \frac12 \|u\|_\mu^2 \\ &\geq \frac12 \left( 1 - \frac{ \mu}{c_{N,\alpha} H_{N,\alpha} } \right) c_{N,\alpha} \iint_{\R^N \times \R^N} \frac{|u(x)-u(y)|^2}{|x-y|^{N+\alpha}} \, dx \, dy + \frac12 \int_{\R^N} V(x) u^2 \, dx \\ &\geq \frac12 \left( 1 - \frac{\mu}{ c_{N,\alpha}H_{N,\alpha} } \right) \|u\|^2
\end{align*}
and the conclusion follows. \qed
\end{proof}

\begin{remark}
$\| \cdot \|_\mu$ is an equivalent norm on $H^{\alpha / 2} (\R^N)$ for $0 \leq \mu < \mu^*$. Indeed -- observe that the bilinear form
\begin{align*}
Q_\mu (u,v) &= c_{N,\alpha} \iint_{\R^N \times \R^N} \frac{|(u(x)-u(y))(v(x)-v(y))}{|x-y|^{N+\alpha}} \, dx \, dy + \int_{\R^N} V(x)uv \, dx \\ &\quad - \frac{\mu}{2} \int_{\R^N} \frac{uv}{|x|^\alpha} \, dx
\end{align*}
is positive-definite and symmetric. Hence it induces a norm $\|u\|_\mu = \sqrt{Q_\mu(u,u)}$ and in view of Lemma \ref{norm-eqv} it is equivalent to $\| \cdot \|$.
\end{remark}

\begin{proof}[Theorem \ref{existence-hardy}]
Rewrite the functional $\cJ$ in the form
$$
\cJ (u) = \frac{1}{2} \|u\|_\mu^2 - \cI (u),
$$
where $\cI (u) :=  \int_{\R^N} F(x,u) + \frac{1}{q} K(x) |u|^q \, dx$. We can easily check that (J1)--(J4) in Theorem \ref{abstract} are satisfied on the space $\left( H^{\alpha / 2} (\R^N), \| \cdot \|_\mu \right)$. Hence there is a bounded minimizing sequence $ \{ u_n \} \subset \cN$ such that
$$
\cJ'(u_n) \to 0, \quad \cJ(u_n) \to c,
$$
where
$$
c = \inf_{\cN} \cJ > 0.
$$
Denote
$$
\cJ_\infty (u) := \cJ(u) - \frac12 \int_{\R^N} V_\loc (x) u^2 \, dx
$$
and let $\cN_\infty$ be the corresponding Nehari manifold. Suppose that $V_{\loc} \equiv 0$. Then $\cJ = \cJ_{\infty}$ and in view of Theorem \ref{ThDecomposition} we have
$$
c +o(1) = \cJ(u_n) \to \cJ (u_0) + \sum_{k=1}^\ell \cJ (w^k) + \frac{\mu}{2} \sum_{k=1}^\ell \int_{\R^N} \frac{|w^k |^2}{|x|^\alpha} \, dx \geq \cJ(u_0) + \ell c.
$$
If $u_0 \neq 0$ we obtain $c \geq (\ell + 1 ) c$ and $\ell = 0$, thus $u_0$ is a ground state solution. If $u_0 = 0$ we obtain $\cJ(u_0) = \cJ(0)=0$ and $c \geq \ell c$. Since $c > 0$, we have $\ell = 1$ and $w^k \neq 0$ is a ground state.

\noindent Suppose now that $V_{\loc} < 0$. Denote $c_{\infty} = \inf_{\cN_{\infty}} \cJ_{\infty} > 0$. As in \cite{BieganowskiMederski} we can show that $c_{\infty} > c$. Indeed, take a critical point $u_\infty \neq 0$ of $\cJ_\infty$ such that $\cJ_\infty (u_\infty) = c_\infty$. Let $t > 0$ be such that $tu_\infty \in \cN$. While $V(x) < V_{\loc} (x)$, we obtain
$$
c_\infty = \cJ_\infty (u_\infty) \geq \cJ_\infty(t u_\infty) > \cJ (t u_\infty) \geq c > 0.
$$
Then
\begin{align*}
c + o(1) = \cJ(u_n) &\to \cJ(u_0) + \sum_{k=1}^\ell \cJ_{\infty} (w^k) + \frac{\mu}{2} \sum_{k=1}^\ell \int_{\R^N} \frac{|w^k |^2}{|x|^\alpha} \, dx \\ &\geq \cJ(u_0) + \ell c_{\infty}.
\end{align*}
Since $c_{\infty} > c$, we have $\ell = 0$ and $u_0 \neq 0$ is a ground state solution.
\qed
\end{proof}

We have also the following nonexistence fact.

\begin{theorem}[{\cite[Theorem 1.2]{Bieganowski-hardy}}]\label{nonex-hardy}
Suppose that (V1), (V2), (K), (F1)--(F4) are satisfied, $\mu < 0$ and
$$
V_\loc (x) > \frac{\mu}{|x|^\alpha} \ \mathrm{for} \ \mathrm{a.e.} \ x \in \R^N \setminus \{0\},
$$
in particular $V_\loc > 0$ can be considered. Then \eqref{eq:1.1} has no ground state solutions.
\end{theorem}

\begin{proof}[Theorem \ref{nonex-hardy}]
Suppose that $u \in \cN$ is a ground state solution of \eqref{eq:1.1}. Denote by $\cJ_{\per}$ the energy functional with $\mu = 0$ and $V_{\loc} \equiv 0$, and let $\cN_{\per}$ be the corresponding Nehari manifold. Let $t > 0$ be such that $tu \in \cN_0$. Then
\begin{align*}
c_{\per} := \inf_{\cN_{\per}} \cJ_{\per} \leq \cJ_{\per} (t u) &= \cJ (t u) - \frac{1}{2} \int_{\R^N} V_{\loc}(x) |tu|^2 \, dx + \frac{\mu}{2} \int_{\R^N} \frac{|tu|^2}{|x|^\alpha} \, dx \\ &< \cJ(tu) \leq \cJ(u) =:c.
\end{align*}
Fix $z \in \mathbb{Z}^N$ and $u_{\per} \in \cN_{\per}$. Then there is $t(z) > 0$ such that $t(z) \tau_z u_{\per} \in \cN$. Observe that
\begin{align*}
&\quad \frac{1}{|t(z)|^{q-2}} \left( \|u_0\|^2 - \mu \int_{\R^N} \frac{| \tau_z u_0|^2}{|x|^\alpha} \, dx \right) \\ &= \frac{1}{|t(z)|^{q}} \int_{\R^N} f(x, t(z)u_0) t(z) u_0 \, dx - \frac{1}{|t(z)|^{q}} \int_{\R^N} K(x) |t(z)|^q |u_0|^q \, dx \\
&\geq \frac{1}{|t(z)|^q} \int_{\R^N} q F(x, t(z)u_0) \, dx -  \int_{\R^N} K(x) |u_0|^q \, dx \\
&= q \int_{\R^N} \frac{ F(x, t(z)u_0)}{|t(z)|^q} \, dx -  \int_{\R^N} K(x) |u_0|^q \, dx.
\end{align*}
The right hand side tends to $\infty$ as $t(z) \to \infty$, while the left hand side stays bounded. Hence $\{ t(z) \} \subset \R$ is bounded if $|z| \to \infty$. Hence, take any sequence $\{ z_n \} \subset \mathbb{Z}^N$ such that $|z_n| \to \infty$. We may assume that $t(z_n) \to t_0$ as $n \to \infty$ and $t_0 \geq 0$. Observe that, in view of Lemma \ref{hardyLemma},
\begin{align*}
&\quad \cJ_{\per} (u_{\per}) = \cJ_{\per} (\tau_z u_{\per} ) \geq \cJ_{\per} (t(z) \tau_z u_{\per}) \\
&= \cJ (t(z) \tau_z u_{\per}) - \frac{|t(z)|^2}{2} \int_{\R^N} V_{\loc} (x) |\tau_z u_{\per}|^2 \, dx + \frac{\mu |t(z)|^2}{2} \int_{\R^N} \frac{|\tau_z u_{\per}|^2}{|x|^\alpha} \, dx \\
&\geq c - \frac{|t(z)|^2}{2} \int_{\R^N} V_{\loc} (x+z) |u_{\per}|^2 \, dx + \frac{\mu |t(z)|^2}{2} \int_{\R^N} \frac{|\tau_z u_{\per}|^2}{|x|^\alpha} \, dx \\
 &\geq c + o(1).
\end{align*}
Taking infimum over $u_{\per} \in \cN_{\per}$ we obtain $c_{\per} < c \leq c_{\per}$ - a contradiction.
\qed
\end{proof}

\section{Profile decomposition of bounded minimizing sequences}
\label{sec:6}

In this section we present three decomposition results, in the spirit of \cite{JeanjeanTanaka}. While proofs are technical, we referee the reader to cited papers. In both subsections we consider the following assumptions on function $g : \R^N \times \R \rightarrow \R$:
\begin{enumerate}
\item[(G1)] $g(\cdot, u)$ is measurable and $\mathbb{Z}^N$-periodic in $x \in \R^N$, $g(x,\cdot)$ is continuous in $u \in \R$ for a.e. $x \in \R^N$;
\item[(G2)] $g(x,u) = o(u)$ as $|u| \to 0^+$ uniformly in $x \in \R^N$;
\item[(G3)] there exists $2 < r < 2_\alpha^*$ such that $\lim_{|u| \to \infty} g(x,u)/|u|^{r-1} = 0$ uniformly in $x \in \R^N$;
\item[(G4)] for each $a<b$ there is a constant $c>0$ such that  $|g(x,u)|\leq c$ for a.e. $x\in\R^N$ and $a\leq u\leq b$.
\end{enumerate} 
Moreover we denote $G(x,u) = \int_0^u g(x,s) \, ds$.

\subsection{The decomposition with bounded potentials}

We consider the functional $\cJ : H^{\alpha / 2} (\R^N) \rightarrow \R$ of the form
$$
\cJ (u) = \frac{1}{2} \|u\|^2 - \int_{\R^N} G(x,u) \, dx,
$$
where the norm is given by
$$
\|u\|^2 = \left\{ \begin{array}{ll}
\int_{\R^N} |\xi|^\alpha |\hat{u}(\xi)|^2 \, d\xi + \int_{\R^N} V(x) u^2 \, dx, & \quad 0 < \alpha < 2, \\
\int_{\R^N} |\nabla u|^2 + V(x) u^2 \, dx, & \quad \alpha = 2.
\end{array} \right.
$$
Put
$$
\cJ_{\per} (u) = \cJ (u) - \frac12 \int_{\R^N} V_{\loc} (x) u^2 \, dx
$$

\begin{remark}
Observe that $\cJ_\per$ is $\Z^N$-invariant, i.e.
$$
\cJ_\per( \tau_k u) = \cJ_\per (u)
$$
for any $u \in H^{\alpha / 2} (\R^N)$ and $k \in \Z^N$.
\end{remark}

\begin{theorem}[{$\alpha = 2$: \cite[Theorem 4.1]{BieganowskiMederski}, $0<\alpha<2$: \cite[Theorem 3.1]{Bieganowski}}]\label{ThDecomposition-per}
Suppose that ($G1$)--($G4$) and ($V_\alpha 1$) hold. Let $\{ u_n \}$ be a bounded Palais-Smale sequence for $\cJ$. Then passing to a subsequence of $\{ u_n \}$, there exist an integer $\ell > 0$ and sequences $\{y_n^k\} \subset \mathbb{Z}^N$, $w^k \in H^{\alpha / 2} (\R^N)$, $k = 1, \ldots, \ell$ such that:
\begin{enumerate}
\item[(a)] $u_n \rightharpoonup u_0$ and $\cJ' (u_0) = 0$;
\item[(b)] $|y_n^k| \to \infty$ and $|y_n^k - y_n^{k'}| \to \infty$ for $k \neq k'$;
\item[(c)] $w^k \neq 0$ and $\cJ_{per}'(w^k) = 0$ for each $1 \leq k \leq \ell$;
\item[(d)] $u_n - u_0 - \sum_{k=1}^\ell w^k (\cdot - y_n^k) \to 0$ in $H^{\alpha / 2} (\R^N)$ as $n \to \infty$;
\item[(e)] $\cJ (u_n) \to \cJ(u_0) + \sum_{k=1}^\ell \cJ_{per} (w^k)$.
\end{enumerate}
\end{theorem}

\subsection{The decomposition with singular potentials}

We consider the functional $\cJ : H^{\alpha / 2} (\R^N) \rightarrow \R$ of the form
$$
\cJ (u) = \frac{1}{2} \|u\|^2 - \mu \int_{\R^N} \frac{u^2}{|x|^\alpha} \, dx - \int_{\R^N} G(x,u) \, dx,
$$
where the norm is defined by 
$$
\|u\|^2 = c_{N,\alpha} \iint_{\R^N \times \R^N} \frac{|u(x)-u(y)|^2}{|x-y|^{N+\alpha}} \, dx \, dy + \int_{\R^N} V(x) u^2 \, dx,
$$
where
$$
c_{N,\alpha} = \frac{2^\alpha \Gamma \left( \frac{N+\alpha}{2} \right)}{2 \pi^{N/2} \left| \Gamma \left( - \frac{\alpha}{2} \right) \right|}
$$
and $\alpha \in (0,2)$. The norm is associated with the following scalar product
$$
\langle u, v \rangle := c_{N,\alpha} \iint_{\R^N \times \R^N} \frac{|u(x)-v(y)|^2}{|x-y|^{N+\alpha}} \, dx \, dy + \int_{\R^N} V(x) u(x)v(x) \, dx.
$$ 
We will also denote
$$
\cJ_{\infty} (u) = \cJ (u) - \frac12 \int_{\R^N} V_{\loc} (x) u^2 \, dx
$$
and $\mu^* = 2^{\alpha} \left( \frac{\Ga \left( \frac{N+\alpha}{4} \right) }{\Ga \left( \frac{N-\alpha}{4} \right) } \right)^2$.

\begin{remark}
Note that $\cJ_\infty$ is not $\mathbb{Z}^N$-invariant in the sense that
$$
\cJ_\infty (\tau_k u) \neq \cJ_\infty (u), \quad u \in H^{\alpha / 2} (\R^N), \ k \in \mathbb{Z}^N.
$$
In fact,
$$
\cJ_\infty (\tau_k u) - \cJ_\infty (u) = \mu \int_{\R^N} \frac{u^2 - (\tau_k u)^2}{|x|^\alpha} \, dx = \mu \int_{\R^N} \frac{u^2}{|x|^\alpha} + o(1) \quad \mbox{as} \ |k|\to\infty.
$$
\end{remark}

\begin{theorem}[{\cite[Theorem 3.1]{Bieganowski-hardy}}]\label{ThDecomposition}
Suppose that ($G1$)--($G4$) and (V1), (V2) hold and $0 \leq \mu < \mu^*$. Let $\{u_n\}$ be a bounded Palais-Smale sequence for $\cJ$. Then passing to a subsequence of $\{u_n\}$, there exist an integer $\ell > 0$ and sequences $\{y_n^k\} \subset \mathbb{Z}^N$, $w^k \in H^{\alpha / 2} (\R^N)$, $k = 1, \ldots, \ell$ such that:
\begin{enumerate}
\item[(a)] $u_n \rightharpoonup u_0$ and $\cJ' (u_0) = 0$;
\item[(b)] $|y_n^k| \to \infty$ and $|y_n^k - y_n^{k'}| \to \infty$ for $k \neq k'$;
\item[(c)] $w^k \neq 0$ and $\cJ_{\infty}'(w^k) = 0$ for each $1 \leq k \leq \ell$;
\item[(d)] $u_n - u_0 - \sum_{k=1}^\ell w^k (\cdot - y_n^k) \to 0$ in $H^{\alpha / 2} (\R^N)$ as $n \to \infty$;
\item[(e)] $\cJ (u_n) \to \cJ(u_0) + \sum_{k=1}^\ell \cJ_{\infty} (w^k) + \frac{\mu}{2} \sum_{k=1}^\ell \int_{\R^N} \frac{|w^k|^2}{|x|^\alpha} \, dx$.
\end{enumerate}
\end{theorem}

\begin{acknowledgement}
The author was partially supported by the National Science Centre, Poland (Grant No. 2017/25/N/ST1/00531).
\end{acknowledgement}

\end{document}